\documentclass[a4paper,reqno]{amsart}
\usepackage{amsmath}
\usepackage{amssymb}
\usepackage{amsthm}
\usepackage{microtype}
\usepackage{mathrsfs}
\usepackage{graphicx}
\usepackage{textcomp}
\usepackage{booktabs}

\newtheorem{thm}{Theorem}[section]

\newtheorem{lemma}[thm]{Lemma}

\theoremstyle{definition}

\newtheorem{remark}[thm]{Remark}

\numberwithin{equation}{section}

\def\Q{\mathfrak{Q}}

\def\Qk{\Q^{(k)}}

\def\Dk{\mathfrak{D}^{(k)}}

\def\ac{\alpha_c}
\def\eig#1#2{\lambda_{#1,#2}}
\def\eigone#1{\eig{1}{#1}}
\def\eigtwo#1{\eig{2}{#1}}

\def\R{\mathbb{R}}

\let\phi=\varphi
\let\epsilon=\varepsilon

\title{A uniqueness theorem for higher order anharmonic oscillators}

\author{S\o ren Fournais}
\author{Mikael Persson Sundqvist}
\address[S\o ren Fournais]{Aarhus University, Department of
  Mathematics, Ny Munkegade 181, 8000 Aarhus C, Denmark}
\email{fournais@imf.au.dk}
\address[Mikael Persson Sundqvist]{Lund University, Department of Mathematical 
Sciences, Lund, Sweden}
\email{mickep@maths.lth.se}

\begin{document}

\begin{abstract}
We study for $\alpha\in\R$, $k \in {\mathbb N} \setminus \{0\}$ the family of 
self-adjoint operators 
\[
-\frac{d^2}{dt^2}+\Bigl(\frac{t^{k+1}}{k+1}-\alpha\Bigr)^2
\]
in $L^2(\R)$ and show that if $k$ is even then $\alpha=0$ gives the unique 
minimum of the lowest eigenvalue of this family of operators.
Combined with earlier results this gives that for any $k \geq 1$, the lowest 
eigenvalue has a unique minimum as a function of $\alpha$.
\end{abstract}

\maketitle

\section{Introduction}

\subsection{Definition of $\Qk(\alpha)$ and main result}
For any $k \in {\mathbb N}\setminus\{0\}$ and $\alpha\in\mathbb{R}$ we define 
the operator
\begin{equation*}
\Qk(\alpha) = -\frac{d^2}{dt^2}+\Bigl(\frac{t^{k+1}}{k+1}-\alpha\Bigr)^2,
\end{equation*}
as a self-adjoint operator in $L^2({\mathbb R})$. This family of operators is
connected with the study of Schr\"{o}dinger operators with a magnetic field 
vanishing along a curve and with the Ginzburg-Landau theory of superconductivity.
It first appeared in~\cite{mo} (for $k=1$) and was later studied 
in~\cite{hemo1,pakw,heko1,helf,hepe,dora,fope,fope2}.

We denote by $\bigl\{ \lambda_{j,\Qk(\alpha)}\bigr\}_{j=1}^{\infty}$ the 
increasing sequence of eigenvalues of $\Qk(\alpha)$. In particular, 
$\eigone{\Qk(\alpha)}$ is the ground state eigenvalue, and we denote by 
$u_{\alpha}$ the associated positive, $L^2$-normalized eigenfunction.

The main result of the present paper is the following theorem.

\begin{thm}\label{thm:main}
Assume that $k\geq 2$ is an even integer. Then $\eigone{\Qk(\alpha)}$ 
attains a unique minimum at $\alpha=0$.
Moreover, this minimum is non-degenerate.
\end{thm}

\begin{remark}
This extends the previous results and discussions
in~\cite{helf,hepe}, where similar results were obtained for odd $k$. 
The non-degeneracy was proved in~\cite{hepe}. In that paper it was also
shown that Theorem~\ref{thm:main} is valid for large even $k$. The fact that
the minimum is attained at $\alpha=0$ was suggested by numerical computations
done by V. Bonnaillie--No\"{e}l. 
\end{remark}

Combining our 
Theorem~\ref{thm:main} with the results of ~\cite{helf,hepe} we get the 
following complete answer.

\begin{thm}
For any $k \in {\mathbb N}\setminus\{0\}$, the function 
$\alpha \mapsto \eigone{\Qk(\alpha)}$ attains a unique minimum. Moreover, 
this minimum is non-degenerate.
\end{thm}

The paper is organized as follows. In Section~\ref{sec:aux} we give several 
spectral bounds on the first two eigenvalues of $\Qk(\alpha)$. These estimates
are used to prove Theorem~\ref{thm:main} for $2\leq k\leq 68$ in 
Section~\ref{sec:proofsmallk} and for $k\geq 70$ in 
Section~\ref{sec:prooflargek}.

\section{Auxiliary results}
\label{sec:aux}
\subsection{Introduction}
In this section we collect several spectral bounds that will help us in
proving Theorem~\ref{thm:main}. In the following, we assume that $k$ denotes 
a positive even integer.

With the scaling $s=\alpha^{-1/(k+1)}t$ it becomes clear that the form domain
of $\Qk(\alpha)$ is independent of $\alpha$. Thus, we are allowed to use the 
machinery of analytic perturbation theory.

First we note that $\Qk(\alpha)$ and $\Qk(-\alpha)$ are unitarily equivalent
(map $t\mapsto-t$ along with $\alpha\mapsto-\alpha$). This implies that the
function $\alpha\mapsto\eigone{\Qk(\alpha)}$ is even, and hence has a
critical point at $\alpha=0$. It is proved in~\cite{hepe} that this critical
point is a nondegenerate minimum. This also follows from our estimates below.

\begin{lemma}\label{lem:Virial}
If $\ac$ is a critical point of $\eigone{\Qk(\alpha)}$, then 
\[
\int_{-\infty}^{+\infty} 
\Bigl(\frac{t^{k+1}}{k+1}-\ac\Bigr)u_{\ac}(t)^2\,dt = 0
\]
and 
\[
\int_{-\infty}^{+\infty} 
\Bigl(\frac{t^{k+1}}{k+1}-\ac\Bigr)^2u_{\ac}(t)^2\,dt =  \frac{\eigone{\Qk(\ac)}}{k+2}.
\]
\end{lemma}

\begin{proof}[Sketch of proof]
The first identity, usually referred to as the Feynman--Hellmann formula, 
follows from first order perturbation theory,
\[
\frac{\partial}{\partial\alpha}\eigone{\Qk(\alpha)} = -2\int_{-\infty}^{+\infty} 
\Bigl(\frac{t^{k+1}}{k+1}-\alpha\Bigr)u_{\alpha}(t)^2\,dt.
\]
The second is a virial type identity and is proved by 
scaling. We refer to~\cite{hepe} for the details.
\end{proof}

\subsection{Positive second derivative}

A key element in our approach is the following Lemma~\ref{lem:possecdiff}, 
which can be used to rule out local maxima under appropriate estimates on the 
first eigenvalues.

\begin{lemma}[Lemma~2.3 in~\cite{hepe}]
\label{lem:possecdiff}
If $\ac$ is a critical point of $\eigone{\Qk(\alpha)}$ and
\[
\frac{k+2}{k+6}\eigtwo{\Qk(\ac)}>\eigone{\Qk(\ac)}
\]
then
\[
\frac{\partial^2}{\partial\alpha^2}
\eigone{\Qk(\alpha)}\Big|_{\alpha=\ac}>0.
\]
\end{lemma}

We give a sketch of the proof for the sake of completeness.

\begin{proof}[Sketch of proof]
The proof is based on perturbation theory. The second derivative of 
$\eigone{\Qk(\alpha)}$ is given by
\[
\frac{\partial^2}{\partial\alpha^2}\eigone{\Qk(\alpha)}
=
2-4\int_{-\infty}^{+\infty} \Bigl(\frac{t^{k+1}}{k+1}-\alpha\Bigr)u_\alpha
\bigl(\partial_\alpha u_\alpha\bigr)\,dt.
\]
Here
\[
\partial_\alpha u_\alpha = -2(\Qk(\alpha)-\eigone{\Qk(\alpha)})^{-1}
\Bigl(\frac{t^{k+1}}{k+1}-\alpha\Bigr)u_\alpha,
\]
where the inverse is the regularized resolvent. The rest of the proof uses
Lemma~\ref{lem:Virial}, the bound 
\[
\|(\Qk(\ac)-\eigone{\Qk(\ac)})^{-1}\|
\leq (\eigtwo{\Qk(\ac)}-\eigone{\Qk(\ac)})^{-1},
\]
and the Cauchy-Schwarz inequality.
\end{proof}

To apply Lemma~\ref{lem:possecdiff} we need good upper bounds on 
$\eigone{\Qk(\alpha)}$ and lower bounds on $\eigtwo{\Qk(\alpha)}$. These will 
be presented in the sections below.

\subsection{Upper bounds}
We will at several points need upper bounds on the first eigenvalue of
$\Qk(\alpha)$. They are given in this section.

\begin{lemma}
\label{lem:criticalub}
Assume that $\ac$ is a critical point of 
$\alpha\mapsto\eigone{\Qk(\alpha)}$. Then, for all $\alpha\in\R$ it holds
that
\[
\eigone{\Qk(\alpha)}\leq \eigone{\Qk(\ac)}+(\alpha-\ac)^2.
\]
\end{lemma}

\begin{proof}
This follows by inserting the eigenfunction $u_{\ac}$ corresponding to 
$\eigone{\Qk(\ac)}$ of $\Qk(\ac)$
into the quadratic form corresponding to $\Qk(\alpha)$ and using 
Lemma~\ref{lem:Virial}
\end{proof}

\begin{lemma}
\label{lem:trial}
For all $\alpha\geq 0$ it holds that
\[
\eigone{\Qk(\alpha)} \leq \alpha^2 + A_k,
\]
with
\[
A_k=
\begin{cases}
\frac{2^{3/2}}{9}
\bigl(\frac{4\pi^6-210\pi^4+4410\pi^2-26775}{7}\bigr)^{1/4},& k=2,\\
\frac{\pi^2}{4}\frac{k+2}{k+1}
\bigl(\frac{1}{4}(k+1)(2k+3)(2k+4)(2k+5)\bigr)^{-1/(k+2)}, & k\geq 2.\\
\end{cases}
\]
\end{lemma}

\begin{proof}
For $k\geq 4$ we refer to Lemma~3.1 in~\cite{hepe}. For $k=2$ we use the same
idea but with a different trial state. A calculation of the energy of the 
function
\[
u(t)=
\begin{cases}
\frac{2}{\sqrt{3\rho}}\cos^2\bigl(\frac{\pi t}{2\rho}\bigr), &|t|<\rho,\\
0, & |t|\geq \rho,\\
\end{cases}
\]
gives ($\|u\|=1$)
\[
\begin{aligned}
\eigone{\Q^{(2)}(\alpha)}
&\leq
\int_{-\infty}^{+\infty}|u'(t)|^2
+\Bigl(\frac{t^3}{3}-\alpha\Bigr)^2|u(t)|^2\,dt\\
&=
\alpha^2+\frac{\pi^2}{3\rho^2}
+\frac{4\pi^6-210\pi^4+4410\pi^2-26775}{252\pi^6}\rho^6.
\end{aligned}
\]
Minimizing in $\rho$, we get the bound
\begin{equation*}
\label{eq:zerobound}
\eigone{\Q^{(2)}(\alpha)}\leq \alpha^2
+\frac{2^{3/2}}{9}
\Bigl(\frac{4\pi^6-210\pi^4+4410\pi^2-26775}{7}\Bigr)^{1/4}
\leq \alpha^2+ 0.6642,
\end{equation*}
attained for 
\[
\rho=2^{1/4}\pi
\Bigl(\frac{4\pi^6-210\pi^4+4410\pi^2-26775}{7}\Bigr)^{-1/8}
\approx 2.57.
\]
\end{proof}

The upper bound given in Lemma~\ref{lem:trial} is graphed (for $\alpha=0$ and
$2\leq k\leq 70$) in Figure~\ref{fig:lambda1comp} on 
page~\pageref{fig:lambda1comp}.

\begin{lemma}
\label{lem:increasingub}
The function
\[
k\mapsto \frac{\pi^2}{4}\frac{k+2}{k+1}
\Bigl(\frac{1}{4}(k+1)(2k+3)(2k+4)(2k+5)\Bigr)^{-1/(k+2)}
\]
appearing in Lemma~\ref{lem:trial} is increasing for $k\geq 2$. In particular
it is always bounded from above by $\pi^2/4$.
\end{lemma}

\begin{proof}
We will in the proof consider $k$ to be a real variable. Taking the logarithmic
derivative of the expression, we get
\[
\frac{a_3k^3+a_2k^2+a_1k+a_0}{(k+1)(k+2)^2(2k+3)(2k+5)}
\]
with (here we note that each term is increasing with $k$ and thus estimate
from below with $k=2$)
\[
\begin{aligned}
a_3 &= 4 \log (2 (k+1) (k+2) (2 k+3) (2 k+5))-20-8 \log 2\\
  &\geq 4\log 378-20\geq 3.73,\\
a_2 &= 20 \log (2 (k+1) (k+2) (2 k+3) (2 k+5))-108-40 \log 2\\
  &\geq 20 \log 378-108\geq 10.69,\\
a_1 &= 31 \log (2 (k+1) (k+2) (2 k+3) (2 k+5))-189-62 \log 2\\
  &\geq 31\log 378-189 \geq -5.02,\\
a_0 &= 15 \log (2 (k+1) (k+2) (2 k+3) (2 k+5))-107-30 \log 2\\
  &\geq 15\log 378-107\geq -17.98.
\end{aligned}
\]
Now, the polynomial
\[
p(k)=3.73k^3+10.69k^2-5.02k-17.98
\]
satisfies
\[
p(2)\approx 44.58\quad\text{and}\quad p'(k)=11.19k^2+21.38k-5.02.
\]
Since $p'(k)>0$ for $k\geq 2$ we find that $p$ is positive for $k\geq 2$. This 
implies that the function in the statement is increasing. The final part follows
since the limit as $k\to+\infty$ is $\pi^2/4$.
\end{proof}

\subsection{Lower bounds}
To be able to use Lemma~\ref{lem:possecdiff} we need lower bounds on the 
second eigenvalue. The following function will appear in the bounds.
\begin{lemma}
\label{lem:cmaxformula}
It holds that
\begin{equation}
\label{eq:cmaxformula}
\begin{aligned}
h(a)&:=\max_{0<\sigma<1} (1-\sigma^2)^{a/(a+2)}\sigma^{2/(a+2)}(a/2)^{4/(a+2)}\\
&= 2^{-4/(a+2)} a^{(a+4)/(a+2)} (a+1)^{1/(a+2)-1}.
\end{aligned}
\end{equation}
Moreover, $\lim_{a\to+\infty}h(a)=1$.
\end{lemma}

\begin{proof}
Differentiating $(1-\sigma^2)^{a/(a+2)}\sigma^{2/(a+2)}(a/2)^{4/(a+2)}$ with
respect to $\sigma$ gives
\[
\frac{2^{(a-2)/(a+2)}a^{4/(a+2)}\sigma^{-a/(a+2)} 
\bigl(1-\sigma^2\bigr)^{-2/(a+2)} 
\bigl(1-(a+1)\sigma^2\bigr)}{a+2},
\]
with the unique zero (in $0<\sigma<1$) at $\sigma=1/\sqrt{a+1}$. Since the 
function is zero
at the endpoints and positive for $0<\sigma<1$ this must be the maximum.
This proves~\eqref{eq:cmaxformula}

The rest follows by a simple analysis of the right hand side 
of~\eqref{eq:cmaxformula}. The derivative equals
\[
h'(a)=\frac{2^{-4/(a+2)} a^{2/(a+2)} (a+1)^{1/(a+2)-1}
\bigl[a \bigl(4+4\log 2 - 2 \log a-\log(a+1)\bigr)+8\bigr]}{(a+2)^2}.
\]
\end{proof}

\begin{lemma}
\label{lem:commutator}
For all real $\alpha$ and all even $k\geq 2$ it holds that
\[
\Qk(\alpha) \geq h(k)
\biggl[-\frac{d^2}{dt^2}+\Bigl(\frac{t^{k/2}}{k/2}\Bigr)^2\biggr],
\]
where $h$ is the function from Lemma~\ref{lem:cmaxformula}.
\end{lemma}

\begin{proof}
Let $\mathfrak{A}=-i\frac{d}{dt}$ and 
$\mathfrak{B}=\bigl(\frac{t^{k+1}}{k+1}-\alpha\bigr)$. Then
the commutator $[\mathfrak{A},\mathfrak{B}]$ equals
\[
[\mathfrak{A},\mathfrak{B}]=-it^k.
\]
With the Cauchy--Schwarz inequality and the weighted arithmetic-geometric mean 
inequality, we find that (for all $0<\sigma<1$)
\[
\Qk(\alpha)\geq -(1-\sigma^2)\frac{d^2}{dt^2}+\sigma t^k 
= -(1-\sigma^2)\frac{d^2}{dt^2}+\sigma(k/2)^2 \Bigl(\frac{t^{k/2}}{k/2}\Bigr)^2.
\]
Scaling the variable and invoking Lemma~\ref{lem:cmaxformula} gives the result.
\end{proof}

\begin{lemma}
\label{lem:lb21}
Let $h$ be the function in Lemma~\ref{lem:cmaxformula}. For all real $\alpha$ 
and all even $k\geq 2$ it holds that
\[
\eigtwo{\Qk(\alpha)}\geq B_k,
\]
with
\[
B_k=h(k)\frac{3^{2k/(k+2)}(k+2)}{2^{(2k-2)/(k+2)}k^{(k+4)/(k+2)}}
=\frac{3^{\frac{2k}{k+2}}(k+2)}{2^{\frac{2k+2}{k+2}}(k+1)^{\frac{k+1}{k+2}}}.
\]
\end{lemma}

\begin{proof}
Let $T>0$. We use the estimate
\[
\Bigl(\frac{t^{k/2}}{k/2}\Bigr)^2 
\geq \frac{2}{k}T^{k-2}t^2-\frac{2k-4}{k^2}T^k,
\]
valid for all $t\in\R$.
Comparing with the harmonic oscillator, and using Lemma~\ref{lem:commutator}, 
we get the required estimate for the second eigenvalue. The optimal choice
of $T$ is
\begin{equation}
\label{eq:optT}
T=\Bigl(\frac{3\sqrt{2k}}{4}\Bigr)^{2/(k+2)}.
\end{equation}
\end{proof}

The lower bound of $\eigtwo{\Qk(\alpha)}$ in Lemma~\ref{lem:lb21} will tend 
to $9/4$ as $k\to+\infty$, which compared to the limit $\pi^2/4$ for the
first eigenvalue is not good enough. Our next aim is to improve this lower 
bound on $\eigtwo{\Qk(\alpha)}$ for large $k$.

\begin{lemma}
\label{lem:betterl2}
Assume that $k\geq 70$ is even and $\alpha\in\R$. Then
\[
\eigtwo{\Qk(\alpha)}\geq \widetilde{B}_k,
\]
with
\[
\widetilde{B}_k=
\frac{\sqrt{5}-1}{2}
\Biggl(\frac{\pi-\arctan
\Bigl(\sqrt{\frac{(\pi/1.1)^2}{1.1^{70}-(\pi/1.1)^2}}\Bigr)}{1.1}\Biggr)^2
\geq 4.719.
\]
\end{lemma}

\begin{proof}
We first do the commutator estimate
\[
\Qk(\alpha)\geq -(1-\sigma^2)\frac{d^2}{dt^2}+\sigma t^k 
= \frac{\sqrt{5}-1}{2}\Bigl(-\frac{d^2}{dt^2}+t^k\Bigr),
\]
where $\sigma$ in the latter step is chosen to be $\frac{\sqrt{5}-1}{2}$. Next
we note that the second eigenvalue of
\[
-\frac{d^2}{dt^2}+t^k
\]
in $L^2(\R)$ equals the first eigenvalue of the operator
\[
-\frac{d^2}{dt^2}+t^k
\]
in $L^2(\R^+)$ with Dirichlet condition at $t=0$. Let $T>1$. Then
\[
-\frac{d^2}{dt^2}+t^k \geq \Dk:=-\frac{d^2}{dt^2}+T^k\chi_{\{t>T\}},
\]
where we, again, impose a Dirichlet condition at $t=0$.
Here $\chi_D$ denotes the characteristic function of the set $D$.
Let us estimate the first eigenvalue $\eigone{\Dk}$ of $\Dk$. 
Clearly
\[
\eigone{\Dk}\leq \Bigl(\frac{\pi}{T}\Bigr)^2,
\]
which is what one gets considering $(0,T)$ and imposing a Dirichlet condition 
at $t=T$. The ground state of $\Dk$ is given by (in the rest of this proof we 
write $\lambda=\eigone{\Dk}$)
\[
u(t)=
\begin{cases}
c_1\sin(\sqrt{\lambda}t), & 0\leq t\leq T,\\
c_2e^{-\omega t}, & t\geq T,
\end{cases}
\]
where
\[
-\omega^2+T^k=\lambda
\]
and where we have the gluing conditions at $t=T$:
\[
\begin{aligned}
c_1\sin(\sqrt{\lambda}T)&=c_2 e^{-\omega T}\quad\text{and}\\
c_1\sqrt{\lambda}\cos(\sqrt{\lambda}T)&=-c_2 \omega e^{-\omega T}.
\end{aligned}
\]
This gives the equation (in $\sqrt{\lambda}$)
\[
\tan(\sqrt{\lambda}T)=-\frac{\sqrt{\lambda}}{\omega}\quad\text{i.e.}\quad
\tan(\pi-\sqrt{\lambda}T)=\frac{\sqrt{\lambda}}{\omega},
\]
which has a unique solution in the interval 
$\frac{\pi}{2T}<\sqrt{\lambda}<\frac{\pi}{T}$. We think of $T>1$ and $k$ 
large, so that $\sqrt{\lambda}/\omega$ is small, and get
\[
\frac{\sqrt{\lambda}}{\omega}=\sqrt{\frac{\lambda}{T^k-\lambda}}
\leq \sqrt{\frac{(\pi/T)^2}{T^k-(\pi/T)^2}}.
\]
And so by monotonicity
\[
\pi-\sqrt{\lambda}T\leq 
\arctan\biggl(\sqrt{\frac{(\pi/T)^2}{T^k-(\pi/T)^2}}\biggr),
\]
i.e.
\[
\lambda\geq 
\Biggl(\frac{\pi-\arctan
\Bigl(\sqrt{\frac{(\pi/T)^2}{T^k-(\pi/T)^2}}\Bigr)}{T}\Biggr)^2.
\]
Now, without optimizing, we find that with $T=1.1$ and $k\geq 70$ it holds
that
\[
\eigtwo{\Qk(\alpha)}\geq \frac{\sqrt{5}-1}{2}\lambda 
\geq \frac{\sqrt{5}-1}{2}
\Biggl(\frac{\pi-\arctan
\Bigl(\sqrt{\frac{(\pi/1.1)^2}{1.1^{70}-(\pi/1.1)^2}}\Bigr)}{1.1}\Biggr)^2
\geq 4.719.
\]
\end{proof}

We will also need lower bounds on $\eigone{\Qk(\alpha)}$ for large $\alpha$.
This is the content of the following two Lemmas.

\begin{lemma}
\label{lem:lb32}
For $\alpha\geq 3/2$ and even $k\geq 2$ it holds that
\begin{equation}
\label{eq:minbound}
\eigone{\Qk(\alpha)}\geq C_k,
\end{equation}
with
\[
C_k=
\min\Biggl(\Bigl(\frac{3}{2}-\frac{1}{k+1}\Bigr)^2,
\frac{\frac{3}{2}(k+1)-1}{(k+1)
\bigl(\bigl(\frac{3}{2}(k+1)\bigr)^{1/(k+1)}-1\bigr)}\Theta_0\Biggr).
\]
In particular, if $2\leq k\leq 68$ it holds that
\[
\eigone{\Qk(\alpha)}>\eigone{\Qk(0)}
\]
for all $\alpha\geq 3/2$.
\end{lemma}

\begin{proof}
First we note that the potential $\bigl(\frac{t^{k+1}}{k+1}-\alpha\bigr)^2$ is 
decreasing for all $t<1$ (in fact for all $t<((k+1)\alpha)^{1/(k+1)}$), and thus 
it is greater than $(1/(k+1)-3/2)^2$ for all $t<1$ and all 
$\alpha\geq 3/2$.

For $t\geq1$ and $\alpha\geq 3/2$, we estimate
\begin{align*}
\Bigl(\frac{t^{k+1}}{k+1}-\alpha\Bigr)^2 
& = 
\frac{1}{(k+1)^2}
\Biggl(\sum_{j=0}^{k}t^{k-j}\bigl(\alpha(k+1)\bigr)^{1/(k+1)}\Biggr)^2
\Bigl(t- \bigl(\alpha(k+1)\bigr)^{1/(k+1)}\Bigr)^2\\
&\geq \frac{1}{(k+1)^2} 
\biggl(\frac{\frac32(k+1)-1}{\bigl(\frac32(k+1)\bigr)^{1/(k+1)}-1}\biggr)^2
\Bigl(t- \bigl(\alpha(k+1)\bigr)^{1/(k+1)}\Bigr)^2.
\end{align*}
Here we used that the expression in the big sum is increasing
both in $t$ and in $\alpha$, and then applied the formula for a geometric sum.

Thus, comparing with the minimum of the potential for $t<1$ and with the 
de~Gennes operator for $t\geq 1$ we conclude~\eqref{eq:minbound}.

The last part follows by comparing the upper bound in Lemma~\ref{lem:trial}
with the just obtained lower bound (and using the fact that $\Theta_0>0.59$ 
which is known from~\cite{bono}). This is done in Figure~\ref{fig:lambda1comp}.
\begin{figure}[h]
\centering
\includegraphics[width=0.7\textwidth]{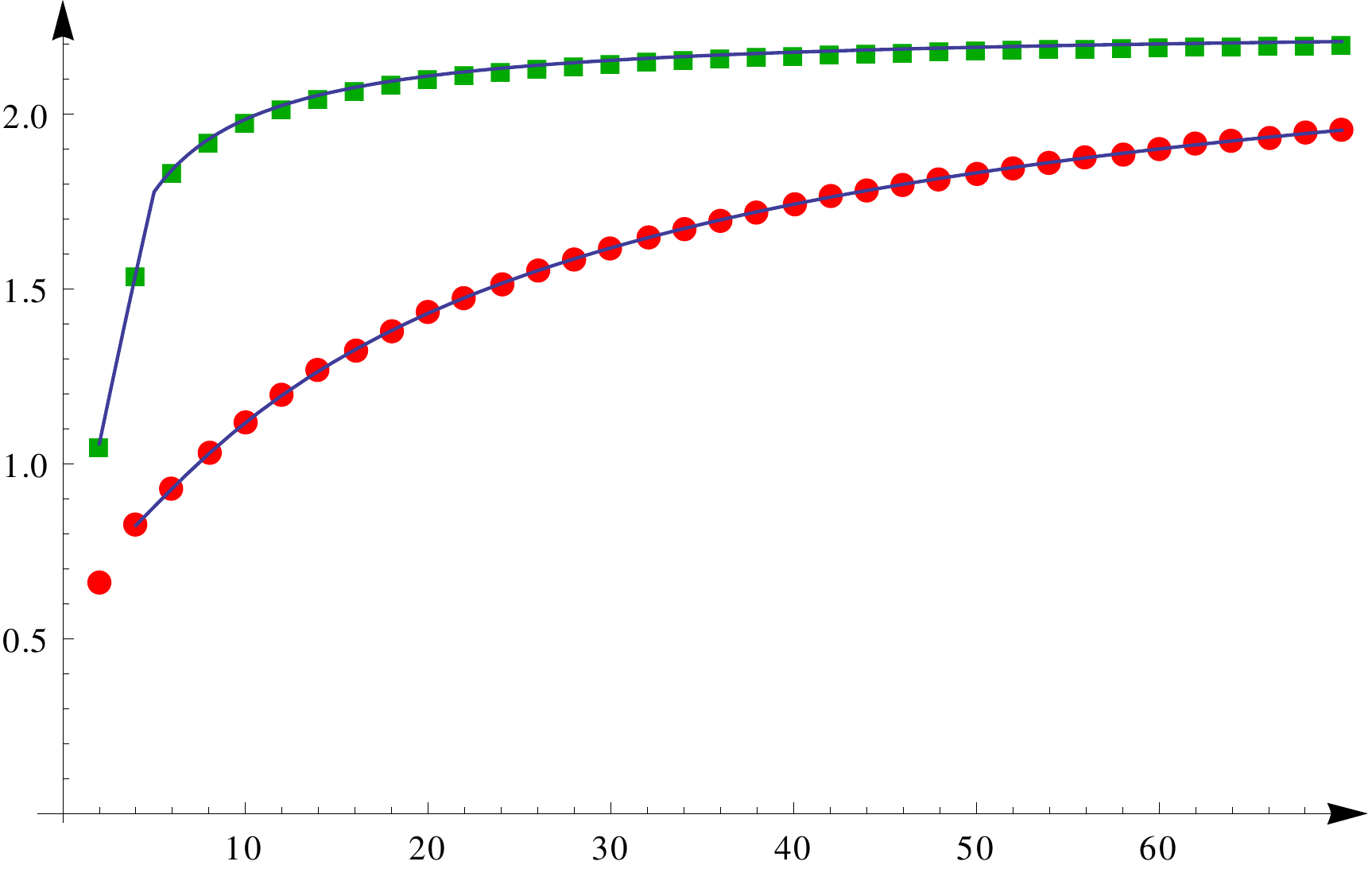}
\caption{The disks are the upper bounds on $\eigone{\Qk(0)}$ from 
Lemma~\ref{lem:trial} as a function of $k$. The squares 
are the lower bounds on $\eigone{\Qk(\alpha)}$, $\alpha\geq 3/2$, as a function 
of $k$, from Lemma~\ref{lem:lb32}.}
\label{fig:lambda1comp}
\end{figure}
\end{proof}

We need a better bound for large $k$ than the one given in Lemma~\ref{lem:lb32}.
We use instead $\alpha=2.8$ as lower bound and find that

\begin{lemma}
\label{lem:lb3}
For $\alpha\geq 2.8$ it holds that
\[
\eigone{\Qk(\alpha)}\geq
\min\Biggl(\Bigl(2.8-\frac{1}{k+1}\Bigr)^2,
\frac{2.8(k+1)-1}{(k+1)
\bigl(\bigl(2.8(k+1)\bigr)^{1/(k+1)}-1\bigr)}\Theta_0\Biggr).
\]
For $k\geq 70$ the first term is the smallest one, i.e.
\[
\eigone{\Qk(\alpha)}\geq \Bigl(2.8-\frac{1}{k+1}\Bigr)^2
\geq \Bigl(2.8-\frac{1}{71}\Bigr)^2
\geq 7.76
.
\]
In particular $\eigone{\Qk(\alpha)}$ cannot obtain its global minimum for 
$\alpha\geq 2.8$.
\end{lemma}

\begin{proof}
The proof is exactly the same as the proof of Lemma~\ref{lem:lb32}. 
The second statement follows from
noticing that the second term in the minimum is increasing and that its value
at $k=70$ is
\[
\frac{2.8(70+1)-1}{(70+1)
\bigl(\bigl(2.8(70+1)\bigr)^{1/(70+1)}-1\bigr)}\Theta_0
\geq 21.2,
\]
while the first term in the minimum is less than $2.8^2=7.84$.

The last statement follows by using Lemma~\ref{lem:increasingub} to conclude 
that the upper bound on $\eigone{\Qk(0)}$ in Lemma~\ref{lem:trial} is less 
than $\pi^2/4$ for all $k$. Since $\pi^2/4$ is less than $7.76$ we are done.
\end{proof}

\section{Proof of Theorem~\ref{thm:main} for $2\leq k\leq 68$}
\label{sec:proofsmallk}

\begin{lemma}
\label{lem:zerotoa}
For each even $k$, $2\leq k\leq 68$, let
\[
\alpha^*=\sqrt{\frac{k+2}{k+6}B_k-A_k},
\]
where $A_k$ is the upper bound on $\eigone{\Qk(0)}$
from Lemma~\ref{lem:trial} and $B_k$ is the lower bound on 
$\eigtwo{\Qk(\alpha)}$ from Lemma~\ref{lem:lb21}.
Then, $\alpha\mapsto\eigone{\Qk(\alpha)}$ has
no critical point in the interval $0<\alpha<\alpha^*$.
\end{lemma}

\begin{proof}
Assume, to get a contradiction, that $0<\alpha_c<\alpha^*$ is a critical point.
Then, invoking Lemma~\ref{lem:trial} and the definition
of $\alpha^*$ above, we find that
\[
\eigone{\Qk(\ac)}\leq A_k+\ac^2 <A_k+(\alpha^*)^2=\frac{k+2}{k+6}B_k \leq 
\frac{k+2}{k+6}\eigtwo{\Qk(\ac)},
\]
which by Lemma~\ref{lem:possecdiff} implies that $\ac$ is a non-degenerate 
local minimum. Hence all critical 
points in $0<\alpha<\alpha^*$ must be non-degenerate local minimums. Now we
know that zero is a non-degenerate local minimum. Since there cannot be
more than one such in a row we get a contradiction.
\end{proof}

\begin{lemma}
\label{lem:zeroto2a}
With $k$ and $\alpha^*$ as in the previous Lemma it holds that 
$\eigone{\Qk(\alpha)}$ cannot attain its global minimal value in the
interval $[\alpha^*,2\alpha^*)$.
\end{lemma}

\begin{proof}
Assume, to get a contradiction, that we have one $\alpha_c$ in this interval
where we have have a global minimum. Then, in particular, 
$\eigone{\Qk(\alpha_c)}\leq\eigone{\Qk(0)}$. Thus, combining again
Lemmas~\ref{lem:possecdiff} and~\ref{lem:criticalub} we find that any critical
point in $[\alpha^*,\alpha_c)$ must be a non-degenerate minimum. However, by
the previous Lemma we know that there are no critical points in $(0,\alpha^*)$,
and so again we would have two non-degenerate minimums in a row. Since that is
not possible we get a contradiction.
\end{proof}

\begin{lemma}
\label{lem:alb}
Assume that $2\leq k\leq 68$ is even. Denote by
\[
\alpha^{**}=\frac{3}{2}-\sqrt{C_k-A_k},
\]
where, again, $A_k$ is the upper bound on $\eigone{\Qk(0)}$ 
from Lemma~\ref{lem:trial} and $C_k$ is the lower bound on 
$\eigone{\Qk(\alpha)}$ from 
Lemma~\ref{lem:lb32}. 

If $\alpha>\alpha^{**}$
then $\eigone{\Qk(\alpha)}$ cannot attain its global minimum.
\end{lemma}

\begin{proof}
First we note that if $\alpha\geq 3/2$ then 
$\eigone{\Qk(\alpha)}\geq \eigone{\Qk(0)}$ by Lemma~\ref{lem:lb32}.
Assume, to get a contradiction, that $\eigone{\Qk(\alpha)}$ attains its global
minimum for a $\alpha^{**}<\alpha_c<3/2$. Then, by 
Lemma~\ref{lem:criticalub} it holds that
\[
\begin{aligned}
\eigone{\Qk(3/2)}&\leq \eigone{\Qk(\alpha_c)}+(\alpha_c-3/2)^2\\
&\leq \eigone{\Qk(0)}+(\alpha^{**}-3/2)^2\\
&< A_k+C_k-A_k=C_k.
\end{aligned}
\]
But this contradicts Lemma~\ref{lem:lb32}.
\end{proof}

The proof of Theorem~\ref{thm:main} is completed for $2\leq k\leq 68$ by 
ploting $2\alpha^*$ and 
$\alpha^{**}$ and noting that $2\alpha^{*}>\alpha^{**}$ for these $k$. This
is done in Figure~\ref{fig:completeproof}.
\begin{figure}
\centering
\includegraphics[width=0.7\textwidth]{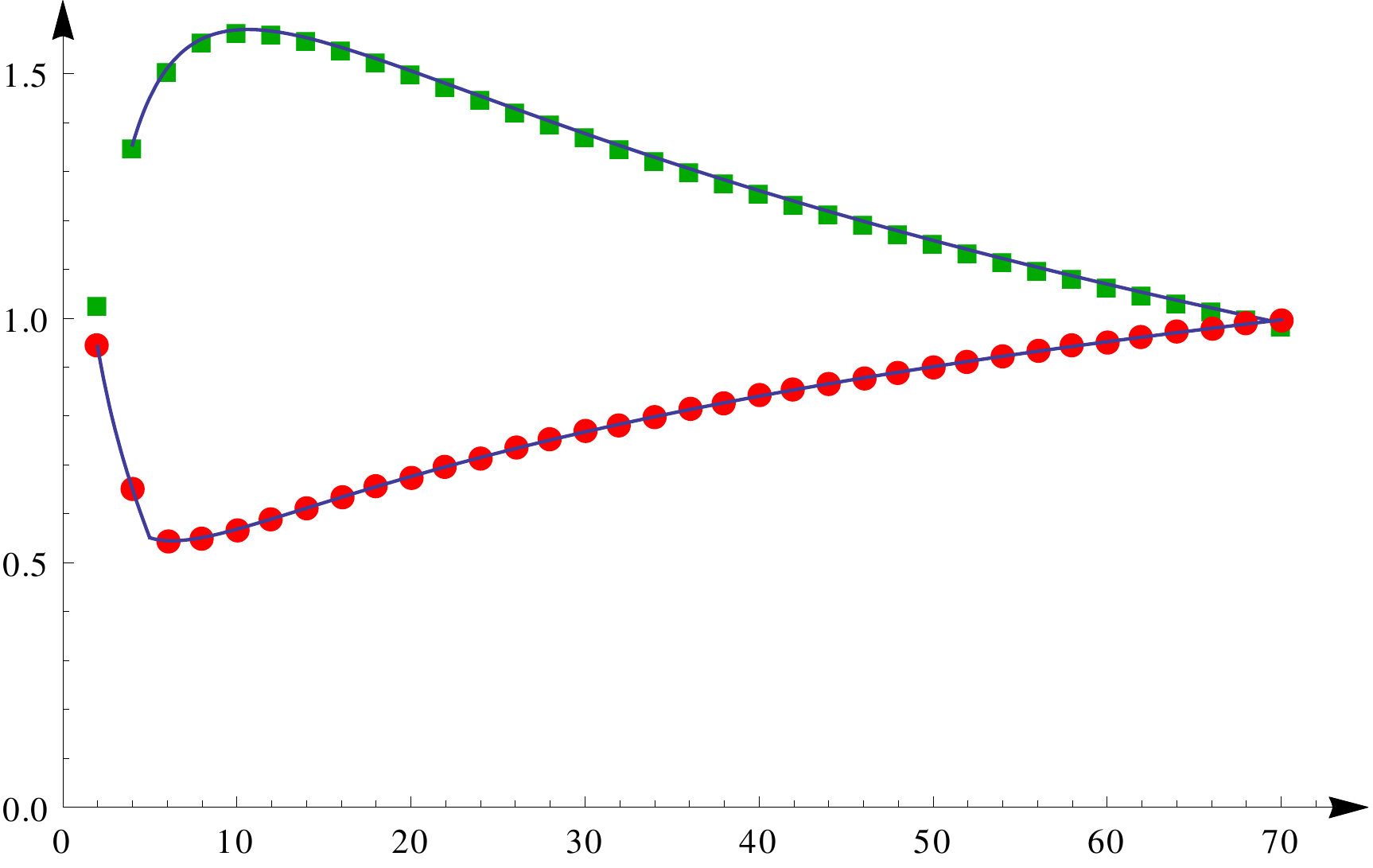}
\caption{The disks are $\alpha^{**}$ from Lemma~\ref{lem:alb}. The squares
are $2\alpha^*$, where $\alpha^*$ is defined in Lemma~\ref{lem:zerotoa}.}
\label{fig:completeproof}
\end{figure}

\section{Proof of Theorem~\ref{thm:main} for $k\geq 70$}
\label{sec:prooflargek}

\begin{lemma}
\label{lem:alpha1largek}
Assume that $k\geq 70$. Then $\eigone{\Qk(\alpha)}$ cannot have its global
minimum for $0<\alpha<2.83$.
\end{lemma}

\begin{proof}
This follows the same lines as the proofs of Lemmas~\ref{lem:zerotoa} 
and~\ref{lem:zeroto2a}. We let
\[
\alpha^* = \sqrt{\frac{k+2}{k+6}\widetilde{B}_k-A_k},
\]
where $A_k$ is the upper bound on 
$\eigone{\Qk(0)}$ from Lemma~\ref{lem:trial} (which is increasing in $k$ by
Lemma~\ref{lem:increasingub}) and $\widetilde{B}_k$ is the lower bound on 
$\eigtwo{\Qk(\alpha)}$ from Lemma~\ref{lem:betterl2}.
For $k\geq 70$ we note that
\[
2\alpha^*\geq 2\sqrt{\frac{72}{76}\times 4.719-\frac{\pi^2}{4}}\geq
2.83.
\]
\end{proof}

Combining this result with Lemma~\ref{lem:lb3} we find that 
$\eigone{\Qk(\alpha)}$ cannot have its minimum attained for $\alpha>0$. This
proves Theorem~\ref{thm:main}.

\section*{Acknowledgements}
SF was partially supported by the Lundbeck
  Foundation, the Danish Natural Science Research Council and the European 
Research Council under the
 European Community's Seventh Framework Program (FP7/2007--2013)/ERC grant
 agreement  202859.

\bibliographystyle{abbrv}
\bibliography{FPS}

\def\cprime{$'$}
\begin{thebibliography}{10}

\bibitem{bono}
V.~Bonnaillie-No{\"e}l.
\newblock Harmonic oscillators with {N}eumann condition of the half-line.
\newblock {\em Commun. Pure Appl. Anal.}, 11(6):2221--2237, 2012.

\bibitem{dora}
N.~Dombrowski and N.~Raymond.
\newblock Semiclassical analysis with vanishing magnetic fields.

\bibitem{fope}
S.~Fournais and M.~Persson.
\newblock Strong diamagnetism for the ball in three dimensions.
\newblock {\em Asymptot. Anal.}, 72(1-2):77--123, 2011.

\bibitem{fope2}
S.~Fournais and M.~Persson~Sundqvist.
\newblock Lack of diamagnetism and the {L}ittle--{P}arks effect.
\newblock {\em arXiv:XXXX:XXXX}, 2013.

\bibitem{helf}
B.~Helffer.
\newblock The {M}ontgomery model revisited.
\newblock {\em {Colloquium Mathematicum, volume in honor of A. Hulanicki}},
  118(2):391--400, 2010.

\bibitem{heko1}
B.~Helffer and Y.~A. Kordyukov.
\newblock Spectral gaps for periodic {S}chr\"odinger operators with
  hypersurface magnetic wells.
\newblock In {\em Mathematical results in quantum mechanics}, pages 137--154.
  World Sci. Publ., Hackensack, NJ, 2008.

\bibitem{hemo1}
B.~Helffer and A.~Morame.
\newblock Magnetic bottles in connection with superconductivity.
\newblock {\em J. Funct. Anal.}, 185(2):604--680, 2001.

\bibitem{hepe}
B.~Helffer and M.~Persson.
\newblock Spectral properties of higher order anharmonic oscillators.
\newblock {\em Journal of Mathematical Sciences}, 165(1), 2010.

\bibitem{mo}
R.~Montgomery.
\newblock Hearing the zero locus of a magnetic field.
\newblock {\em Comm. Math. Phys.}, 168(3):651--675, 1995.

\bibitem{pakw}
X.-B. Pan and K.-H. Kwek.
\newblock Schr\"odinger operators with non-degenerately vanishing magnetic
  fields in bounded domains.
\newblock {\em Trans. Amer. Math. Soc.}, 354(10):4201--4227 (electronic), 2002.

\end{thebibliography}

\end{document}